\documentclass{amsart}
\usepackage{amsmath}
\usepackage{amsfonts}
\usepackage{amssymb}
\usepackage{xcolor}
\usepackage{graphicx}
\newtheorem{Theorem}{Theorem}[section]

\newtheorem{Lemma}[Theorem]{Lemma}
\newtheorem{Proposition}[Theorem]{Proposition}

\newtheorem{Definition}[Theorem]{Definition}
\newtheorem{rem}[Theorem]{Remark}
\newtheorem{example}[Theorem]{Example}

\newcommand{\R}{\mathbb R}
\newcommand{\K}{\mathbb{K}}

\newcommand{\N}{\mathbb N}

\newcommand{\C}{\mathcal{C}}

\newcommand{\F}{\mathcal{F}}
\newcommand{\p}{\mathcal{P}}

\newcommand{\rel}{\mathcal{R}}

\newcommand{\mcc}{\mathcal{C}}

\title[Connections and covariant derivatives on vector pseudo-bundles]{On the geometry of diffeological vector pseudo-bundles and infinite dimensional vector bundles: automorphisms, connections and covariant derivatives}

\author{Jean-Pierre Magnot}

\address{$^*$  Univ. Angers, CNRS, LAREMA, SFR MATHSTIC, F-49000 Angers, France
	\\ and \\  Lyc\'ee Jeanne d'Arc, \\ Avenue de Grande Bretagne, \\ 63000 Clermont-Ferrand, France}

\email{magnot@math.cnrs.fr}
\email{jean-pîerr.magnot@ac-clermont.fr}

\begin{document}

	\begin{abstract}
		We consider here the category of diffeological vector pseudo-bundles, and study a possible extension of classical differential geometric tools on finite dimensional vector bundles, namely, the group of automorphisms, the frame bundle, the space of connection 1-forms and the space of covariant derivatives. Substential distinctions are highlighted in this generalized framework, among which the non-isomorphism between connection 1-forms and covariant derivatives. Applications not only include finite dimensional examples with singularities, but also infinite dimensional vector bundles.
	\end{abstract}
	
	\maketitle
	
	\tableofcontents
	\section{Introduction}
	The extension of the classical framework of differential geometry to examples of interest is a long-term goal since Sophus Lie, Elie Cartan and Emmy Noether, methods got deeper and deeper insight in order to deal with more and more complex frameworks. 
	
	We consider here the setting of vector pseudo-bundles, a terminology which first appears in  \cite{GKO2000} to our knowledge but the objects under consideration are of interest since a much longer time. The point is to consider ``bundles'' which typical fiber can depend on the basepoint, in the category of vector spaces. Ekatarina Pervova \cite{pervova,pervova2,pervova2017,pervova2018} gave a new impulsion of the subject, motivated by the various generalizations of the notions  of tangent spaces and cotangent spaces in the diffeological setting, which all produce natural examples of such ``bundles'' with a vector space as fiber, which may differ in the same arcwise component. The simple example of the cross in the plane is explained as an example in \cite{Ma2020-3}. In her works, Pervova considered finite dimensional vector pseudo-bundles in the category of diffeological spaces. This is only a small class of diffeological spaces, while diffeological spaces include most infinite dimensional geometric object. 
	In fact, diffeological spaces, first defined by Souriau and Chen, is actually a not so badly developped framework for differential geometry, see the textbook \cite{Igdiff} which is now a not-so-up-to-date treatise of the basic concepts of the subject. This category includes, as subcategories, all the other categories of interest for differential geometry: finite dimensional manifolds, Hilbert, Banach or Fr\'echet manifolds, orbifolds. This is the category that we choose to deal with and that produces, as already mentioned, many examples of vector pseudo-bundles when considering tangent functors. 
	
	This work starts with the necessary preliminaries on diffeologies, in which we give some examples that appear to us of interest, some of them are already disseminated in the recent literature, and few of them are new to our knowledge. 
	Then it is necessary to extend first the notion of group of automorphims to the category of diffeological vector pseudo-bundles. The classical short exact sequence that decomposes the group of automorphisms of a finite dimensional vector bundle, see e.g.  \cite{Neeb2008}, extends to the group of automorphisms of a diffeological vector pseudo-bundle. It enables us to reduce the diffeology $\p$ of the diffeological vector pseudo-bundle $E$ to a weaker diffeology $\p_{Aut(E)}$ which separates in different connected components the various typical fibers of $E,$ that is, a connected component of $E$ in  $\p_{Aut(E)}$ is a diffeological vector bundle (with typical fiber, but maybe without local trivialization).
	
	With this new diffeology, we show that the definition of the frame bundle and of connections 1-forms extends quite straightway, as well as the bundle of vector space endomorphisms $End(E).$ This result appears to us quite natural, since the vector pseudo-bundle is now cut into connected components which are vector bundles. This result gives a geometric interpretation of a known result on diffeomorphisms of a space with singularities: isolated singularities are fixed points in the connected component of the identity map. But a problem still remains: recover a diffeology on all these objects that depends on $\p$ and not on the restricted diffeology $\p_{Aut(E)}.$ In this perspective, we only manage to ``link again'' $End(E)$ by bridging together the disconnected parts of $\p_{Aut(E)}$ through the plots of $\p,$ and we remark that we can define a generalized notion of covariant derivatives, that for an affine space parametrized by $End(E),$ while connection 1-forms of the frame bundle have very different aspects. 
	
	After all these descriptions, an outlook section seems very necessary to us, in order to highlight how our new objects fit with existing constructions. In this last section, we do not have the ambition to be exhaustive, but only to show how this generalized framework fits with a selected family of examples of interest for us. Namely, we describe the structures brought by this new approach on infinite dimensional vector bundles beyond Banach of ILB vector bundles, and we analyze some properies of some affine connections on a cone that e have defined in this work.   
	\section{Preliminaries on diffeologies and vector pseudo-bundles}
	This section provides the necessary background on diffeology and vector pseudo-bundles, mostly following \cite{MR2019,Ma2020-3} in the exposition while the main reference for diffeologies remains \cite{Igdiff}. A complementary non exhaustive bibliography is \cite{BIgKWa2014,BN2005,BT2014,CN,CSW2014,CW2014,Don,DN2007-1,DN2007-2,IgPhD,Leandre2002,pervova,Wa}.
	
	\subsection{Basics of Diffeology}\label{ss:diffeology}
	
	In this subsection we review the basics of the theory of diffeological spaces; in particular, their definition, categorical properties, as well as their induced topology. {{} The main idea of diffeologies (and Fr\"olicher spaces defined shortly after) is to replace the atlas of a classical manifold by other intrinsic objects that enable to define smoothness of mappings in a safe way, considering manifolds as a restricted class of examples. Many such settings have been developped independently. We choose these two settings because they carry nice properties such as cartesian closedness, carrying the necessary fundamental properties of e.g. calculus of variations, and also because they are very easy to use in a differential geometric way of thinking. The fundamental idea of these two settings consists in defining families of smooth maps, with mild conditions on them that ensure technical features of interest. }
	
	\begin{Definition}[Diffeology] \label{d:diffeology}
		Let $X$ be a set.  A \textbf{parametrisation} of $X$ is a
		map of sets
		$p \colon U \to X$ where $U$ is an open subset of Euclidean space (no fixed dimension).  A \textbf{diffeology} $\p$ on $X$ is a set of
		parametrisations satisfying the following three conditions:
		\begin{enumerate}
			\item (Covering) For every $x\in X$ and every non-negative integer
			$n$, the constant function $p\colon \R^n\to\{x\}\subseteq X$ is in
			$\p$.
			\item (Locality) Let $p\colon U\to X$ be a parametrisation such that for
			every $u\in U$ there exists an open neighbourhood $V\subseteq U$ of $u$
			satisfying $p|_V\in\p$. Then $p\in\p$.
			\item (Smooth Compatibility) Let $(p\colon U\to X)\in\p$.
			Then for every $n$, every open subset $V\subseteq\R^n$, and every
			smooth map $F\colon V\to U$, we have $p\circ F\in\p$.
		\end{enumerate}
		A set $X$ equipped with a diffeology $\p$ is called a
		\textbf{diffeological space}, and is {{} denote}d by $(X,\p)$.
		When the diffeology is understood, we will drop the symbol $\p$.
		The parametrisations $p\in\p$ are called \textbf{plots}.
	\end{Definition}
	
	{{}
		\noindent\textbf{Notation.} We recall that $\N^* = \{n \in \N \, | \, n \neq 0\}$ and that $\forall m \in \N^*, \N_m = \{1,...,m\} \subset \N.$}
	
	\begin{Definition}[Diffeologically Smooth Map]\label{d:diffeolmap}
		Let $(X,\p_X)$ and $(Y,\p_Y)$ be two diffeological
		spaces, and let $F \colon X \to Y$ be a map.  Then we say that $F$ is
		\textbf{diffeologically smooth} if for any plot $p \in \p_X$,
		$$F \circ p \in \p_Y.$$
	\end{Definition}
	
	Diffeological spaces with diffeologically smooth maps form a category. This category is complete and co-complete, and forms a quasi-topos (see \cite{BH}).
	
	\begin{Proposition}
		\cite{Sou,Igdiff} Let $(X',\p)$ be a diffeological space,
		and let $X$ be a set. Let $f:X\rightarrow X'$ be a map.
		We define $f^*(\p)$ the \textbf{pull-back diffeology} as {{} $$f^*( \p)= \left\{ p: O_p \rightarrow X \, |   f \circ p \in \p \right\}. $$ }
	\end{Proposition}
	
	\begin{Proposition} \cite{Sou,Igdiff} Let $(X,\p)$ be a diffeological space,
		and let $X'$ be a set. Let $f:X\rightarrow X'$ be a map.
		We define $f_*(\p)$ the \textbf{push-forward diffeology} as the coarsest (i.e. the smallest for inclusion) among the diffologies
		on $X'$, which contains $f \circ \p.$ \end{Proposition}  
	
	\begin{Definition}
		Let $(X,\p)$ and $(X',\p')$
		be two diffeological spaces.  A map
		$f : X \rightarrow X'$
		is  called  a
		subduction
		if  $\p' = f_*(\p).$ \end{Definition}
	.
	In particular, we have the following constructions.
	
	\begin{Definition}[Product Diffeology]\label{d:diffeol product}
		Let $\{(X_i,\p_i)\}_{i\in I}$ be a family of diffeological spaces.  Then the \textbf{product diffeology} $\p$ on $X=\prod_{i\in I}X_i$ contains a parametrisation $p\colon U\to X$ as a plot if for every $i\in I$, the map $\pi_i\circ p$ is in $\p_i$.  Here, $\pi_i$ is the canonical projection map $X\to X_i$. 
	\end{Definition}
	
	In other words, in last definition, $\p = \cap_{i \in I} \pi_i^*(\p_i)$ and each $\pi_i$ is a subduction. 
	
	\begin{Definition}[Subset Diffeology]\label{d:diffeol subset}
		Let $(X,\p)$ be a diffeological space, and let $Y\subseteq X$.  Then $Y$ comes equipped with the \textbf{subset diffeology}, which is the set of all plots in $\p$ with image in $Y$.
	\end{Definition}
	If $X$ is a smooth manifolds, finite or infinite dimensional, modelled on a complete locally convex topological vector space, we define the \textbf{nebulae diffeology}
	$$\p(X) = \left\{ p \in C^\infty(O,X) \hbox{ (in the usual sense) }| O \hbox{ is open in } \R^d, d \in \N^* \right\}.$$
	\subsection{Fr\"olicher spaces}
	\begin{Definition} $\bullet$ A \textbf{Fr\"olicher} space is a triple
		$(X,\F,\mcc)$ such that
		
		- $\mcc$ is a set of paths $\R\rightarrow X$,
		
		- A function $f:X\rightarrow\R$ is in $\F$ if and only if for any
		$c\in\mcc$, $f\circ c\in C^{\infty}(\R,\R)$;
		
		- A path $c:\R\rightarrow X$ is in $\mcc$ (i.e. is a \textbf{contour})
		if and only if for any $f\in\F$, $f\circ c\in C^{\infty}(\R,\R)$.
		
		\vskip 5pt $\bullet$ Let $(X,\F,\mcc)$ et $(X',\F',\mcc')$ be two
		Fr\"olicher spaces, a map $f:X\rightarrow X'$ is \textbf{differentiable}
		(=smooth) if and only if one of the following equivalent conditions is fulfilled:
		\begin{itemize}
			\item $\F'\circ f\circ\mcc\subset C^{\infty}(\R,\R)$
			\item $f \circ \C \subset \C'$
			\item $\F'\circ f \subset  \F$ 
		\end{itemize}
	\end{Definition}
	
	Any family of maps $\F_{g}$ from $X$ to $\R$ generate a Fr\"olicher
	structure $(X,\F,\mcc)$, setting \cite{KM}:
	
	- $\mcc=\{c:\R\rightarrow X\hbox{ such that }\F_{g}\circ c\subset C^{\infty}(\R,\R)\}$
	
	- $\F=\{f:X\rightarrow\R\hbox{ such that }f\circ\mcc\subset C^{\infty}(\R,\R)\}.$
	
	One easily see that $\F_{g}\subset\F$. This notion will be useful
	in the sequel to describe in a simple way a Fr\"olicher structure.
	A Fr\"olicher space carries a natural topology,
	which is the pull-back topology of $\R$ via $\F$. In the case of
	a finite dimensional differentiable manifold, the underlying topology
	of the Fr\"olicher structure is the same as the manifold topology. In
	the infinite dimensional case, these two topologies differ very often.
	
	Let us now compare Fr\"olicher spaces with diffeological spaces, with the following diffeology {{}$\p_\infty(\F)$} called "nebulae":
	{{}
		{Let }$O${ be an open subset of a Euclidian space; } $$\p_\infty(\F)_O=
		\coprod_{p\in\N}\{\, f : O \rightarrow X; \, \F \circ f \subset C^\infty(O,\R) \quad \hbox{(in
			the usual sense)}\}$$
		and 
		$$ \p_\infty(\F) = \bigcup_O \p_\infty(\F)_O,$$
		where the latter union is extended over all open sets $O \subset \R^n$ for $n \in \N^*.$ 
	}
	With this construction, we get a natural diffeology when
	$X$ is a Fr\"olicher space. In this case, one can easily show the following:
	\begin{Proposition} \label{Frodiff} \cite{Ma2006-3} 
		Let$(X,\F,\mcc)$
		and $(X',\F',\mcc')$ be two Fr\"olicher spaces. A map $f:X\rightarrow X'$
		is smooth in the sense of Fr\"olicher if and only if it is smooth for
		the underlying nebulae diffeologies. \end{Proposition}
	
	Thus, we can also state intuitively:
	\vskip 12pt
	\begin{tabular}{ccccc}
		smooth manifold  & $\Rightarrow$  & Fr\"olicher space  & $\Rightarrow$  & Diffeological space\tabularnewline
	\end{tabular}
	\vskip 12pt
	With this construction, any complete locally convex topological vector space is a diffeological vector space, that is, a vector space for which addition and scalar multiplication is smooth. The same way, any finite or infinite dimensional manifold $X$ has a nebulae diffeology, which fully determines smooth functions from or with values in $X.$We now finish the comparison of the notions of diffeological and Fr\"olicher 
	space following mostly \cite{Ma2006-3,Wa}:
	
	\begin{Theorem} \label{compl-fro}
		Let $(X,\p)$ be a diffeological space. There exists a unique Fr\"olicher structure
		$(X, \F_\p, \mcc_\p)$ on $X$ such that for any Fr\"olicher structure $(X,\F,\mcc)$ on $X,$ these two equivalent conditions are fulfilled:
		
		(i)  the canonical inclusion is smooth in the sense of Fr\"olicher $(X, \F_\p, \mcc_\p) \rightarrow (X, \F, \mcc)$
		
		(ii) the canonical inclusion is smooth in the sense of diffeologies $(X,\p) \rightarrow (X, \p_\infty(\F)).$ 
		
		\noindent Moreover, $\F_\p$ is generated by the family 
		$$\F_0=\lbrace f : X \rightarrow \R \hbox{ smooth for the 
			usual diffeology of } \R \rbrace.$$
		{{} We call \textbf{Fr\"olicher completion} of $\p$ the Fr"olicher structure $(X, \F_\p, \mcc_\p).$}
	\end{Theorem}
	

	
	
	A deeper analysis of these implications has been given
	in \cite{Wa}. The next remark is inspired on this work; it is based on \cite[p.26, Boman's theorem]{KM}.
	\begin{rem}
		We notice that the set of contours $\C$ of the Fr\"olicher space
		$(X,\F,\C)$ \textbf{does not} give us a diffeology, because a diffelogy
		needs to be stable under restriction of domains. In the case of paths in
		$\C$ the domain is always $\R.$ However, $\C$ defines a ``minimal diffeology''
		$\p_1(\F)$ whose plots are smooth parameterizations which are locally of the
		type $c \circ g,$ where $g \in \p_\infty(\R)$  and $c \in \C.$ Within this setting,
		a  map $f : (X,\F,\C) \rightarrow (X',\F',\C')$ is smooth if and only if it is smooth  $(X,\p_\infty(\F)) \rightarrow (X',\p_\infty(\F')) $ or equivalently smooth  .$(X,\p_1(\F)) \rightarrow (X',\p_1(\F')) $ 
	\end{rem}
	We apply the results on product diffeologies to the case of Fr\"olicher spaces and we derive very easily, (compare with e.g. \cite{KM}) the following:
	
	\begin{Proposition} \label{prod2} Let $(X,\F,\C)$
		and $(X',\F',\C')$ be two Fr\"olicher spaces equipped with their natural
		diffeologies $\p$ and $\p'$ . There is a natural structure of Fr\"olicher space
		on $X\times X'$ which contours $\C\times\C'$ are the 1-plots of
		$\p\times\p'$. \end{Proposition}
	
	We can even state the result above for the case of infinite products;
	we simply take cartesian products of the plots or of the contours.
	We also remark that given an algebraic structure, we can define a
	corresponding compatible diffeological structure. For example, a
	$\R-$vector space equipped with a diffeology is called a
	diffeological vector space if addition and scalar multiplication
	are smooth (with respect to the canonical diffeology on $\R$), see \cite{Igdiff,pervova,pervova2}. An
	analogous definition holds for Fr\"olicher vector spaces. Other
	examples will arise in the rest of the text.
	
	\begin{rem} \label{comp}
		Fr\"olicher, $c^\infty$ and Gateaux smoothness are the same notion
		if we restrict to a Fr\'echet context, see \cite[Theorem 4.11]{KM}.
		Indeed, for a smooth map $f : (F, \p_1(F)) \rightarrow \R$ defined
		on a Fr\'echet space with its 1-dimensional diffeology, we have
		that $\forall (x,h) \in F^2,$ the map $t \mapsto f(x + th)$ is
		smooth as a classical map in $\C^\infty(\R,\R).$ And hence, it is
		Gateaux smooth. The converse is obvious.
	\end{rem}
	
	\subsection{Quotient and subsets}
	
	We give here only the results that will be used in the sequel.

	We have now the tools needed to describe the diffeology on a quotient:
	
	\begin{Proposition} \label{quotient} let $(X,\p)$ b a diffeological
		space and $\rel$ an equivalence relation on $X$. Then, there is
		a natural diffeology on $X/\rel$, {{} denote}d by $\p/\rel$, defined as
		the push-forward diffeology on $X/\rel$ by the quotient projection
		$X\rightarrow X/\rel$. \end{Proposition}
	
	Given a subset $X_{0}\subset X$, where $X$ is a Fr\"olicher space
	or a diffeological space, we can define on subset structure on $X_{0}$,
	induced by $X$.
	
	$\bullet$ If $X$ is equipped with a diffeology $\p$, we can define
	a diffeology $\p_{0}$ on $X_{0},$ called \textbf{subset
		diffeology} \cite{Sou,Igdiff} setting \[ \p_{0}=\lbrace p\in\p
	\hbox{ such that the image of }p\hbox{ is a subset of
	}X_{0}\rbrace.\]

	$\bullet$ If $(X,\F,\C)$ is a Fr\"olicher space, we take as a generating
	set of maps $\F_{g}$ on $X_{0}$ the restrictions of the maps $f\in\F$.
	In that case, the contours (resp. the induced diffeology) on $X_{0}$
	are the contours (resp. the plots) on $X$ which image is a subset
	of $X_{0}$.
	\begin{example}
		Let $X$ be a diffeological space. Let us {{} denote} by $$X^\infty = \left\{ (x_n)_{n \in \N} \in X^\N \, | \, \{n \, | x_n \neq 0 \} \hbox{ is a finite set}\right\} $$ Then this is a diffeological space, as a subset of $X^\N.$ The quotient space $X^\N / X^\infty$ is a diffeological space related to the settings of non standard analysis, see e.g. \cite{AFHKL1986}.
	\end{example}
	\subsection{Vector pseudo-bundles}

	{{} Let us now have a precise look at the notion of fiber bundle in classical (finite dimensional) fiber bundles. Fiber bundles, in the context of smooth finite dimensional manifolds, are defined by \begin{itemize}
			\item a smooth manifold  $E$ called total space
			\item a smooth manifold  $X$ called base space
			\item a smooth submersion $\pi: E \rightarrow X$ called fiber bundle projection
			\item a smooth manifold $F$ called typical fiber, because $\forall x \in X,  \pi^{-1}(x)$ is a smooth submanifold of $E$ diffeomorphic to $F.$
			\item a smooth atlas on $X,$ with domains $U \subset X$ such that $\pi^{-1}(U)$ is an open submanifold of $E$ diffeomorphic to $U \times F.$ We the get a system of local trivializations of the fiber bundle.
		\end{itemize}
		By the way, in order to be complete, a smooth fiber bundle should be the quadruple data $(E,X,F,\pi)$ (because the definition of $\pi$ and of $X$ enables to find systems of local trivializations). For short, this quadruple setting is often {{} denote}d by the projection map $\pi: E\rightarrow X.$
		
		There exists some diffeological spaces which carry no atlas, so that, the condition of having a system of smooth trivializations in a generalization of the notion of fiber bundles is not a priori necessary, even if this condition, which is additional, enables interesting technical aspects \cite[pages 194-195]{MW2017}. So that, in a general setting, we do not need to assume the existence of local trivializations.}
	Now, following \cite{pervova}, in which the ideas from \cite[last section]{Sou} have been devoloped to vector spaces, the notion of quantum structure has been introduced in \cite{Sou} as a generalization of principal bundles, and the notion of vector pseudo-bundle in \cite{pervova}.The common idea consist in the description of fibered objects made of a total (diffeological) space $E,$ over a diffeological space $X$ and with a canonical smooth {{} bundle} projection $\pi: E \rightarrow X$ such as, $\forall x \in X,$ $\pi^{-1}(x)$ is endowed with a (smooth) algebraic structure, but for which we do not assume the existence of a system of local trivialization. 
	\begin{enumerate}
		\item For a diffeological vector pseudo-bundle, the fibers $\pi^{-1}(x)$ are assumed diffeological vector spaces, i.e. vector spaces where addition and multiplication over a diffeological field of scalars (e.g. $\R$ or $\mathbb{C}$) is smooth. We notice that \cite{pervova} only deals with finite dimensional vector spaces.
		\item For a so-called ``structure quantique'' (i.e. ``quantum structure'') following the terminology of \cite{Sou}, a diffeological group $G$ is acting on the right, smoothly and freely on a diffeological space $E$. The space of orbits $X=E/G$ defines the base of the quantum structure $\pi: E \rightarrow X,$ which generalize the notion of principal bundle by not assuming the existence of local trivialization. In this picture, each fiber $\pi^{-1}(x)$ is isomorphic to $G.$
	\end{enumerate}
	From these two examples, we can generalize the picture. 
	\begin{Definition}\label{pseu-fib}
		Let $E$ and $X$ be two diffeological spaces and let $\pi:E\rightarrow X$ be a smooth surjective map. Then $(E,\pi,X)$ is a \textbf{diffeological fiber pseudo-bundle} if and only if  $\pi$ is a subduction. 
	\end{Definition}
	{{} Let us precise that we do not assume that there exists a typical fiber, in coherence with Pervova's diffeological vector pseudo-bundles. We
	}
	can give the following definitions along the lines of \cite{Ma2020-3}:
	\begin{Definition}
		Let $\pi:E\rightarrow X$ be a diffeological fiber pseudo-bundle. Then:
		\begin{enumerate}
			\item Let $\mathbb{K}$ be a diffeological field. $\pi:E\rightarrow X$ is a \textbf{diffeological $\mathbb{K}-$vector pseudo-bundle} if there exists: 
			\begin{itemize}
				\item a smooth fiberwise map $.\, :\mathbb{K} \times E \rightarrow E,$
				\item a smooth fiberwise map $+:E^{(2)} \rightarrow E$ where $$E^{(2)} = \coprod_{x \in X} \{(u,v) \in E^2\, | \, (u,v)\in \pi^{-1}(x)\}$$ equipped by the pull-back diffeology of the canonical map $E^{(2)} \rightarrow E^2,$
			\end{itemize}
			such that $\forall x \in X, $ $(\pi^{-1}(x),+,.)$ is a diffeological $\mathbb{K}-$vector bundle. We say that $E$ is a \textbf{ diffeological vector bundle} is $E$ is a diffeological vector pseudo-bundle which fibers are all isomorphic as diffeological vector spaces. 
			\item $\pi:E\rightarrow X$ is a \textbf{Souriau quantum structure} if it is a diffeological principal pseudo-bundle with diffeological gauge (pseudo-)bundle $X\times G \rightarrow X.$ 
		\end{enumerate}
	\end{Definition}
	\subsection{Tangent and cotangent spaces}
	The review \cite{GMW2023} actually identifies five main definitions of tangent space, that generalize the classical tangent space of a finite dimensional manifold from one among the two following starting points
	\begin{itemize}
		\item The tangent space of a finite dimensional manifold is defined by 1-jets of paths.
		\item  The tangent space of a finite dimensional manifold is defined by pointwise derivations of $\R-$valued maps. 
		\end{itemize}
	When replacing ``finite dimensional manifold'' by ``diffeological space'', these two notions do not coincide. Even more, unlike in the $c^\infty-$setting \cite{KM}, the natural mapping $$`` \hbox{1-jets} \quad \longrightarrow \quad \hbox{pointwise derivations}''$$ may not be injective, and the space of 1-jets at one point may be only a cone, and not a vector space. For this reason, five generalizations are actually studied for different ways of application. 
	
	 The \textbf{internal tangent cone} extends straightway the definition of the tangent space of a smooth manifold by germs of paths (compare with the kinematic tangent space in \cite{KM}). This defines the internal tangent cone first described for Fr\"olicher spaces in \cite{DN2007-1}. 
	 This approach seems sufficient for application purpose \cite{GW2021}.  
		 
		But there are other tangent spaces in the category of diffeological spaces: the diff-tangent space which is defined in \cite{Ma2020-3} and that will be described in next parts of the exposition,
		 the \textbf{internal tangent space} is defined in \cite{He1995,CW2014}, based on germs of paths. {{} This second definition is necessary, and the internal tangent space differ from the internal tangent cone. Indeed, spaces of germs do not carry intrinsically a structure of abelian group. This remark was first formulated in the context of Fr\"olicher spaces, see \cite{DN2007-1}, and see \cite{CW2014} for the generalization to diffeologies. For this reason, one can complete the tangent cone into a vector space, called internal tangent space. This was performed in \cite{CW2014} via mild considerations on colimits in categories. From another viewpoint, the \textbf{external tangent space} is spanned by derivations, and one can define the cone of derivations which are defined by germs of paths \cite{Ma2013} as well as its vector space completion in the space of derivations following \cite{GW2020}.  
	For finite dimensional {{}manifolds} these tangent spaces coincide. For an extended presentation, we refer to the review \cite{GMW2023}.
	
	The definition of the cotangent space $T^*X,$ as well as the definition of the algebra of differential forms $\Omega(X,\K)$ does not carry so many ambiguities and is defined through pull-back on each domain of plots, see \cite{Igdiff} and e.g. \cite{Ma2013} for a comprehensive exposition. 
	\subsection{On examples of interest}
	We give here a non-exhaustive list of examples where one can exhibit structures of diffeological vector bundles, and also diffeological vector pseudo-bundles, that we hope motivating for the reader. Two other examples are developed more extensively at the end of the paper.
\begin{example}
	Le $\tau$ be a triangulation of a smooth paracompact manifold $N.$ The question of the diffeology associated to a triangulation is a long story which starts by Ntumba's PhD thesis, which principal results are published in \cite{Nt2002}, and which continues with e.g. \cite{CW2014-2,Ma2016-2,Ma2020-3}. In any of these works, where slightly different diffeologies are defined for a same triangulation, this is a simple exercise to check that $\Omega^*$ is only a vector pseudo-bundle. A toy example remains again on $KerP,$ with $P(X,Y) = XY$.  
\end{example}
	\begin{example}
		Let $E = \R^\N$ and let $F=C_0(\N,\R)$ be the set of sequences that converge to $0.$ We equip these spaces with the topology of uniform convergence. It is well-known that $F$ has no topological complement in $E,$ and that $F$ acts on $E$ by translation. Then the quotient map $E \rightarrow E/F$ defines a diffeological vector bundle with typical fiber $F.$ The existence of local slices for the $D-$topology of $E/F$ is actually unknown, while this setting furnishes a classical counter-example in the setting of topological vector spaces.
	\end{example}
\begin{example}
 Following \cite{Ma2020-3}, let $X$ be a topological vector space and let $Y$ be a Fr\'echet space which is the completion of $X.$ Let $F: X \rightarrow \R$ be a functional on $X$ and let $C(X,Y)$ be the vector space of sequences in $X$ that converge in $Y. $ Then we define the $Y-$weak solutions of the functional equation $$F(x) = 0$$  as the set $$C_Y = \left\{\lim u_n \, | \, (u_n)\in C(X,Y) \hbox{ and } \lim F(u_n)=0\right\}.$$
 The space $C(X,Y)$ has a dedicated diffeology called Cauchy diffeology which ensures smoothness of the limit. Therefore, the limit map $$lim : C_Y \rightarrow Y$$ defines a diffeological pseudo-fiber bundle with total space $C_Y,$ with base $S_Y = lim C_Y$ and with fibers which are all diffeomorphic to diffeological subspaces of $C_0(X,Y),$ the set $X-$sequences that converge to $0$ in $Y.$  Under mild conditions on the functional $F,$ the fibers can be proved to be vector spaces.
\end{example}

\begin{example}
	Let $N$ be a finite dimensional Riemannian manifold. From \cite{Ee}, the spaces of maps $H^s(S^1,N)$ are Hilbert manifolds when $s>1/2.$ Following \cite{Ma2015-2}, $H^s(S^1,N)$ is a diffeological space for $s \leq 1/2,$ with $C^0-$homotopy type which may depend on the value of $s$ (compare also with \cite{Mir2017}). Therefore, the study of their tangent and cotangent space envolve diffeological vector bundles which actually have no known local trivialization. These objects are intersting in optimization theory as highlighted in \cite{We2017}.
\end{example}
\section{Technical constructions: $G-$diffeology and isomorphisms of vector pseudo-bundles}
For this section, we need to recall the following definitions from \cite{Igdiff}: \begin{Definition}

	Let $(X,\p)$ and $(X',\p')$ be two diffeological spaces. Let $S \subset C^\infty(X,X')$ be a set of smooth maps. The \textbf{{
			functional} diffeology} on $S$ is the diffeology $\p_S$
	made of plots
	$$ \rho : D(\rho) \subset \R^k \rightarrow S$$
	such that, 
	for each $p \in \p, $
	the maps $\Phi_{\rho, p}: (x,y) \in D(p)\times D(\rho) \mapsto \rho(y)(x) \in X'$ are plots of $\p'.$
\end{Definition}

\noindent
With this definition, we have the classical fundamental property for calculus of variations and for composition:

\begin{Proposition} \label{functf}\cite{Igdiff}
	Let $X,Y,Z$ be diffeological spaces
	\begin{enumerate}
		\item
		$$C^\infty(X\times Y,Z) = C^\infty(X,C^\infty(Y,Z)) = C^\infty(Y,C^\infty(X,Z))$$as diffeological spaces equipped with {
			functional} diffeologies.
		\item The composition map 
		$$C^\infty(X,Y) \times C^\infty(Y,Z) \rightarrow C^\infty(X,Z)$$ is smooth.  
	\end{enumerate}
	
\end{Proposition}
\subsection{The $G-$diffeology and its tangent space}
\begin{Definition} \label{pG}
	Let $G$ be a diffeological group, and let $\rho: G \times X \rightarrow X$ be an action on a set $X.$ Then we note by $\p_\rho(X),$ or $\p_G(X)$ when there is no ambiguity, the push-forward diffeology from $G$ on $X.$ 
\end{Definition}
This diffeology separates the orbits of $G$ in $X$, that is, two disjoint orbits of $G$ lie in different connected components of $X.$ Therefore, we can define:
\begin{Definition} \label{TG}
	\begin{itemize}
		\item the Fr\"olicher completion $(X,\F_\rho,\mcc_\rho)$ of $(X,\p_\rho)$
		\item $\Omega^*_\rho(X)$ the de Rham complex of $(X,\F_\rho,\mcc_\rho)$
		\item ${}^\rho TX$ the internal tangent cone of $(X,\p_\rho)$ 
	\end{itemize}
\end{Definition} 

\begin{Proposition}
	${}^\rho TX$ is a diffeological vector pseudo-bundle. 
\end{Proposition}
\begin{proof}
	As ${}^\rho TX$ an internal tangent cone, we only have to prove that each fiber of ${}^\rho TX$ over $X$ is a (diffeological) vector space. Let $x \in X$ and let $(c_1,c_2) \in C^\infty(\R,X)^2$ such that $c_1(0)=c_2(0)=x.$ Then there exists $\epsilon >0$ and two local paths  $(g_1,g_2) \in C^\infty(]-\epsilon,\epsilon[;G )^2$ such that $\forall i \in \N_2, c_i|_{]-\epsilon,\epsilon[} = g_i . x$
	Let $X_i = \partial_t c_i|_{t=0}.$ Then $$X_1+X_2 = X_2 + X_1 = \partial_t (g_1.g_2).x|_{t=0} \in {}^\rho T_xX$$ (in fact, we apply the same arguments as for the tangent space of a diffeological group along the lines of \cite{MR2016}, which expands the arguments of \cite{Les}). 
\end{proof}
\begin{example}[The $Diff-$diffeology and the $Diff-$tangent space.]
	Let $X$ be a Fr\"olicher space and let $G = Diff(X)$ the group of diffeomorphisms of $X.$ Following \cite{Ma2020-3}, the Diff-diffeology and the $Diff-$tangent space coincides with our construction in Definitions \ref{pG} and \ref{TG}.
\end{example}
	\subsection{Isomorphisms of vector pseudo-bundles}
	We now rephrase vector pseudo-bundle morphisms and isomorphisms described in \cite{Wu2021} in a non-categorical vocabulary.
	\begin{Definition}
		Let $\pi:E \rightarrow X$ and $\pi':F \rightarrow Y$ be two diffeological vector pseudo-bundles. $(\phi, \varphi)$ is a morphism of vector pseudo-bundles from $\pi:E \rightarrow X$ to $\pi':F \rightarrow Y$ if: 
		\begin{itemize}
			\item $\phi : E \rightarrow F$ is smooth
			\item $\varphi: X \rightarrow Y$ is smooth
			\item $\pi' \circ \phi = \varphi \circ \pi$
			\item $\forall x \in X, $ the restricted map $\phi|_{\pi^{-1}(x)} :  \pi^{-1}(x) \rightarrow \pi'^{-1}(\varphi(x))$ is linear (and also smooth in the subset diffeologies).  
		\end{itemize}
	\end{Definition}
\begin{rem}
	Following \cite{Wu2021}, if $(\phi, \varphi) \in Aut(E),$ then $\varphi$ is the restriction of $\phi$ to the zero-section of $E.$
\end{rem}
	Therefore, 
	\begin{Proposition}\label{3.2}
			Let $\pi:E \rightarrow X$ and $\pi':F \rightarrow Y$ be two diffeological vector pseudo-bundles. Let $(\phi, \varphi)$ be a morphism of vector pseudo-bundles from $\pi:E \rightarrow X$ to $\pi':F \rightarrow Y.$
	$(\phi,\varphi)$ is an isomorphism of vector pseudo-bundles if and only if $\phi \in Diff(E,F).$ If so, we have that:
	\begin{itemize}
		\item $\varphi \in Diff(X,Y)$
		\item $(\phi^{-1},\varphi^{-1})$ is a morphism of vector pseudo-bundles.
	\end{itemize} 
	\end{Proposition}
\begin{proof}
	Let $x \in X.$ Let us consider $\phi^{-1}\circ \phi|_{\pi^{-1}(x)}.$ Since $\phi^{-1}\circ \phi = Id_E,$ we have $\phi^{-1}\circ \phi|_{\pi^{-1}(x)} = Id_{\pi^{-1}(x)}$ therefore $\phi^{-1}$ is fiberwise linear. 
	
	Moreover, we have that $X$ is isomorphic to the zero-section of $E$ and that $Y$ is isomorphic to the zero-section of $F.$ Therefore, the restriction of $\phi^{-1}$ to the zero-section of $Y$ shows existence and smoothness of $\varphi^{-1}.$ 
	
\end{proof}
We remark the following, extending the notion of homotopy of vector bundles to vector pseudo-bundles:
\begin{Lemma}
	Let $X$ be a diffeological vector space. Then any vector pseudo-bundle $E$ over $X$ is homotopic to the null-vector space $X \times \{0\}$ in the category of vector pseudo-bundles, that it, there exists a homotpoty pseudo-bundle $H$ over $X \times [0;1]$ such that $H |_{X \times \{0\}}$ is isomorphic to $E$ and  $H |_{X \times \{1\}}$ is isomorphic to $X \times \{0\}.$ 
\end{Lemma}

\begin{proof}
	We set $$H = \left(E \times (X \times [0;1])\right) / \sim$$
	where the equivalence relation $\sim$ identifies $E \supset X \ni x  \sim (x,0)\in X \times [0;1].$
	Equipped with the quotient diffeology, this is a vector pseudo-bundle which ends the proof. 
\end{proof}
Therefore, isomorphic vector pseudo-bundles cannot be characterized by homotopies.

	\section{On the group of automorphisms and on the frame bundle of a vector pseudo-bundle }
	Let $\pi : E \rightarrow X$ be a vector pseudo-bundle. We note by $E_x$ the fiber $\pi^{-1}(x),$ for $x \in X.$ We also note by $GL(E)$ the gauge pseudo-bundle of endomorphisms of $E,$ that is, the group of automorphisms of $E$ which decompose fiberwise as an automorphism  of each fiber.   
	\subsection{On the group of automorphisms}
	\begin{Definition}
		We note by $Aut(E)$ the set of automorphisms of the vector pseudo-bundle $E.$
	\end{Definition}
\begin{Theorem}
	$Aut(E)$ is a diffeological group, as a diffeological subgroup of $Diff(E).$ Moreover, there is a short exact sequence 
	$$ 0 \rightarrow GL(E) \rightarrow Aut(E) \rightarrow G_E(X) \rightarrow 0$$
	where $G_E(X)$ is a diffeological subgroup of $Diff(X).$ 
\end{Theorem}
\begin{proof}
	From Proposition \ref{3.2}, we already have that $Aut(E)$ is a diffeological subgroup and that the projection $Aut(E) \rightarrow G_E(X)$ is a morphism of diffeological groups. If $(\phi, Id_X)\in Aut(E),$ then $\phi$ is fiberwise a diffeological vector space isomorphism, which ends the proof. 
\end{proof}
Let us analyze better the group $G_E(X).$
\begin{Lemma} \label{fiber}
	Let $(x,y) \in X^2$ such that $\exists (\phi,\varphi)\in Aut(E),$ $y = g(x).$ Then $\phi|_{E_x}$ is an isomorphism from $E_x$ to $E_y.$
\end{Lemma}
The proof of this easy lemma is straightforward, but has deep consequences, for which proofs are straightforward and to our opinion needless. 

\begin{Proposition} \label{prop:fiberisom}
	If $G_E(X) $ acts transitively on $X,$ then $E$ has a typical fiber, i.e.  $\forall (x,y) \in X^2., E_x$ is isomorphic to $E_y$ as diffeological vector spaces.
\end{Proposition}
We now define a new diffeology on $E,$ that is not the $Aut(E)-$ diffeology in the sense of section 3.1, but  
$$ \p_{Aut(E)} = \p \cap \pi^* (\p_{G_E(X)})$$
	where $\p_{G_E(X)})$ is the diffeology defined by the action of $G_E(X)$ on $X$ along the lines of section 3.1.
\begin{Proposition}
	The vector pseudo-bundle $(E,\p_{Aut(E)})$ has connected components with typical fiber.
\end{Proposition}
Therefore, we can produce the following definitions and properties: 
\begin{Definition}
	Let $X$ be a diffeological space. Let $F$ be a vector pseudo-bundle over $X.$ Then $X$ is $F-$singular if  the vector pseudo-bundle $F$ is not a vector bundle. A $F-$\textbf{singularity} is a point $x \in X$ such that there exists no base of neighborhoods $\{V_k \, | \, k \in K\}$ of $x$ in the $D-$topology such that $$\forall k \in K, \quad \pi^{-1}(V_k) \hbox{ is a vector bundle with typical fiber } \pi^{-1}(x).$$ 
\end{Definition}
We get easily the following lemma as a corollary of Proposition \ref{prop:fiberisom}: 
\begin{Lemma}
	Under the last notations, assume that $G_F(X) = Diff(X).$ Let $\{F_i \, | \, i\in I \}$ be the collection of (non isomorphic) fibers of $F,$ indexed by $I.$ Let 
	$$X_{F_i} = \{x \in X \, | \, \pi^{-1}(x) \hbox{ is isomorphic to } F_i\}.$$
	Then $\forall i \in I, X_{F_I} $ is an invariant subset of any diffeomorphism $\varphi \in Diff(X).$
	
\end{Lemma}
We have here at hand at least two examples of such vector pseudo-bundles $F,$ already developed in the literature:
\begin{enumerate}
	\item The group of diffeomorphisms $Diff(X)$ act on the cotangent space $F=T^*X$ (see \cite{Igdiff} for a definition)
	\item The group of diffeomorphisms $Diff(X)$ acts on germs of smooth paths in $X$ and hence it acts on the internal tangent space $F={}^iTX$ defined in \cite{CW2014}.
\end{enumerate}
For each of these two examples, we have a morphism o diffeological groups
$$ Diff(X) \rightarrow Aut(E)$$ which is fiberwise, that is, with a natural left inverse map $$Aut(E) \rightarrow Diff(X)$$ which is exactly the projection $$Aut(E) \rightarrow G_E(X),$$ which implies that $$G_E(X) = Diff(X).$$
Therefore, when $F = T^*X$ or $F={}^iTX,$ any $F-$singularity classifies the diffeomorphisms, heuristically speaking. In particular
\begin{Theorem}
	In the two previous choices of $F,$ any diffeomorphism $\varphi$ smoothly homotopic to the identity map leaves the connected components of each $X_{F_i}$ globally invariant. In particular, an isolated $F-$singularity is a fixed point for $\varphi.$
\end{Theorem}
\begin{proof}
	Let $\varphi_t$ be a smooth path in $Diff(X)$ such that $\varphi_0 = Id_X$ and $\varphi_1 = \varphi.$ Then the fibers over $\varphi_t(x)$ are isomorphic for $t \in [0;1].$
\end{proof}
\begin{rem}
	Last theorem also characterizes disconnected parts on the $Diff-$diffeology on $X.$ This is not another result, but only a reformulation of last proposition.
\end{rem}
\subsection{Frame bundles for diffeological vector bundles}
We define here the geometric objects related to frame bundles of diffeological vector bundles, along the lines of the classical constructions for frames for finite or infinite dimensional vector bundles, see e.g. \cite{Ma2006}. 
\begin{Definition} \label{frame-aut}
	Let $\pi:E\rightarrow X$ be a diffeological vector space with typical fiber $F.$ Let $x \in X$ and let $E_x$ be the fiber over $x.$
	We define the space
	$$Fr(E_x)$$ as the diffeological space of isomorphisms of diffeological vector spaces from $F$ to $E_x,$
	and we define the \textbf{frame bundle} of $E$ by:
	$$ Fr(E) = \coprod_{x \in X} Fr(E_x).$$
\end{Definition} 
We denote by $GL(F)$ the (diffeological) group of (diffeological) isomorphisms of the diffeological vector space $F.$
With this notation, and considering the compostion of diffeological linear maps, we remark that: 
\begin{itemize}
	\item $GL(F)$ acts on the right on $Fr(E)$ and, $\forall x \in X,$ on $Fr(E_x),$
	\item $GL(E)$ acts on the left on $Fr(E),$ and, $\forall x \in X,$ through the canonical inclusion map $GL(E_x) \hookrightarrow GL(E),$ $GL(E)$ and $GL(E_x)$ are acting on the left on $Fr(E_x),$
	\item The group $Aut(E)$ is acting on the left on $Fr(E).$
\end{itemize}

Then, applying Definition \ref{pG}, we get: 
\begin{Lemma}
	$(Fr(E), \p_{Aut(E)}(Fr(E)))$ is a diffeological space.
\end{Lemma}
With this diffeology, the action of $Aut(E)$ on $Fr(E)$ is smooth, and we moreover have:
\begin{Theorem}
	The diffeological space $Fr(E)$ is:
	\begin{enumerate}
		\item a Souriau quantum structure, with respect to the group $GL(F),$ 
\item a diffeological fiber bundle, with typical fiber $GL(F).$	
\end{enumerate}
\end{Theorem}
\begin{proof}
	Given two frames $f_1$ and $f_2$ in $Fr(E_x)$ we remark that 
	\begin{itemize}
		\item $f_2 \circ f_1^{-1} \in GL(E_x)$
		\item $f_1^{-1} \circ f_2 \in GL(F)$
	\end{itemize}
therefore, it is straightforward to check that, on $Fr(E_x),$ the left action of $GL(E_x)$ and the right action of $GL(F)$ are  free and transitive, which shows (1). 
Moreover, the map $ f \in Fr(E) \mapsto f_1^{-1} \circ f $ identifies $Fr(E_x)$ with $GL(F),$ which shows (2).
\end{proof}
\section{Connection 1-forms versus covariant derivatives}
\subsection{Differential forms with values in vector pseudo-bundles}
Let us consider the de Rham complexes $\Omega(E,\R)$ and $\Omega(M,\R).$ As diffeological subgroups of $Diff(E),$ $Aut(E)$ and $ GL(E)$ are acting $\Omega(E,\R)$ and  the same way, $G_E(X)$ is acting on $\Omega^*(X).$ Let us precise this action. For ant plot $p: U \rightarrow X,$ differential form $\alpha \in \Omega^*(X)$ pulls-back, by definition, to the (classical) differential form $p^*\alpha \in \Omega^*(U).$ Given $g \in Diff(X),$ then $q = g \circ p$ is also a plot of $X$ and hence, $g\in Diff(X)$ transforms  $p^*\alpha$ to $q^*\alpha.$ The same construction holds for $Diff(E), Aut(E), G_E(X)$ for their corresponding target spaces.   
In all this section, and at each time where we appeal to the notions that are developed here, $E$ is equipped wit its $Aut(E)-$diffeology and $M$ with its $G_E(X)-$diffeology. Let $p \in \p_{G_E(M))}.$
Let $O_p$ be the domain of $p$ and let $$\alpha_p: \wedge^nTO \rightarrow E$$
be a smooth map such that, if $u \in O_p$ and $x = p(u),$ then $\alpha_p$ respectricts to a skew-symmetric $n-$linear map from $T^n_uO_p$ to $E_x .$
With these notations,
one can define a $n-$form $\alpha \in \Omega^n(M,E)$ 
A $E-$valued $n-$differential form $\alpha$ on $X$ (noted $\alpha \in \Omega^n(X,E))$ is a map 
$$ \alpha : \{p:O_p\rightarrow X\} \in \p \mapsto \alpha_p $$
(where $\alpha_p$ has the property described before)
such that 

$\bullet$ Let $x\in X.$ $\forall p,p'\in \p$ such that $x\in Im(p)\cap Im(p')$, 
the forms $\alpha_p$ and $\alpha_{p'}$ are of the same order $n.$ 

$\bullet$ Moreover, let $y\in O_p$ and $y'\in O_{p'}.$ If $(X_1,...,X_n)$ are $n$ germs of paths in 
$Im(p)\cap Im(p'),$ if there exists two systems of $n-$vectors $(Y_1,...,Y_n)\in (T_yO_p)^n$ and $(Y'_1,...,Y'_n)\in (T_{y'}O_{p'})^n,$ if $p_*(Y_1,...,Y_n)=p'_*(Y'_1,...,Y'_n)=(X_1,...,X_n),$
$$ \alpha_p(Y_1,...,Y_n) = \alpha_{p'}(Y'_1,...,Y'_n).$$

We note by $$\Omega(X;E)=\oplus_{n\in \mathbb{N}} \Omega^n(X,E)$$ the set of $E-$valued differential forms. 
With such a definition, we feel the need to make two remarks for the reader:

$\bullet$ If there does not exist $n$ linearly independent vectors $(Y_1,...,Y_n)$
defined as in the last point of the definition, $\alpha_p = 0$ at $y.$

$\bullet$ Let $(\alpha, p, p') \in \Omega(X,V)\times \p^2.$ 
If there exists $g \in C^\infty(D(p); D(p'))$ (in the usual sense) 
such that $p' \circ g = p,$ then $\alpha_p = g^*\alpha_{p'}.$ 

\vskip 12pt
\begin{Proposition}
	The set $\p(\Omega^n(X,E))$ made of maps $q:x \mapsto \alpha(x)$ from an open subset $O_q$ of a 
	finite dimensional vector space to $\Omega^n(X,E)$ such that for each $p \in \p,$ $$\{ x \mapsto \alpha_p(x) \} \in C^\infty(O_q, \Omega^n(O_p,E)),$$
	is a diffeology on $\Omega^n(X,E).$  
\end{Proposition}
\begin{rem}
	We have here to warn about the temptation of a straightforward extension of the classical wedge product of differential forms which only exists of $E$ is fiberwise equipped with a smooth fiberwise multiplication. 
	Moreover, the de Rham differential also requires much more attention.
\end{rem}
The following proposition is straightforward  from the definition of the diffeologies of $Aut(E)$ and of $G_E(X).$
\begin{Proposition}
	The group $Aut(E)$ is acting smoothly on $\Omega(X,E)$ 
	the following way: 
	let $\alpha \in \Omega(M,E)$ and let $g=(\phi,\varphi) \in Aut(E),$ 
	we define $g_*\alpha$ on each $p \in \p$ by 
	$$\left(g_*\alpha\right)_p = \phi \circ \alpha_{\varphi^{-1} \circ p}.$$
\end{Proposition}


	
	\subsection{Connections and covariant derivatives  on a vector pseudo-bundle}
	Classicaly, given a principal bundle $P$ with structure group $G$ with Lie algebra $\mathfrak{g},$ connections are 1-forms $\theta \in \Omega^1(P,\mathfrak{g})$ (see \cite{Igdiff,Ma2013}) that are covariant under the right action of $G,$ that is, denoting the right action of $g \in G$ by $R_g,$ we have: 
	$$ (R_g)_* \theta = Ad_{g^{-1}} \theta.$$ In this framework, the   de Rham differential and wedge product can be defined by using classically the Lie gracket of $\mathfrak{g}.$ 
	Here, our principal bundle is $Fr(E),$ which may not have only one structure group sunce the typical fiber of $E$ may vary, depending on the connected components of $(E,\p_{Aut(E)}), $ see Lemma \ref{fiber}. Therefore, for this section, one can restrict the study to vector bundles, that is, vector pseudo-bundles with typical fiber $F.$ We note by $\mathfrak{gl}(F)$ the tangent space at identity of $GL(F)$, and we note by $L(F)$ the space of endomorphisms of the diffeological vector space $F.$
	
	\begin{rem}
		In our framework, we are not sure that $L(E)=\mathfrak{gl}(E),$ while this is a classical fact when $E$ is a finite dimensional trivial vector bundle over a compact manifold. To our knowledge the best framework where one can find such properties when $L(F)$ is an infinite dimensional algebra can be found in \cite{Gl2002,GN2012}. The question of the comparison of $L(F)$ with $\mathfrak{gl}(F)$ remains open in the fully general framework. Moreover, $\mathfrak{gl}(F)$ is not a priori a set of endomorphisms of the diffeological vector space  $F$ since $F$ is not assumed to carry a complete topology which ensures an adequate convergence property. 
	\end{rem}
	 Therefore, one can define, along the lines of the discussion present in e.g. \cite{Ma2006}:
	 \begin{Definition}
	 	We define, for $x \in X,$ $$End_L(E)_x = \left\{ e \in C^\infty(E_x,E_x) \, | \, e \hbox{ is linear}\right\},$$
	 	and let 
	 	$$ End_L(E) = \coprod_{x \in X} End_L(E)_x.$$
	 \end{Definition}
	 \begin{Proposition}
	 $$End_{L}(E) = \left\{ e= f \circ l \circ f^{-1} \, | \,  f \in Fr(E) \hbox{ and } l \in L(F)\right\}$$
	\end{Proposition}
\begin{proof}
	Let $(f,l)\in Fr(E)\times L(F).$ Then $f \circ l \circ f^{-1}$ is smooth and linear by smoothness of the composition in the functional diffeology. Conversly, let $e \in End_L(E)_x, $ and let us fix any $f \in Fr(E)_x,$ then $l = f \circ e \circ f^{-1} \in L(F)$ for the same resaons. 
\end{proof}
\begin{Definition} \label{EndLE}
	We equip $End_L(E)$ with the push-forward diffeology $\p_{L,Aut}$ of the map: 
	$$ ReF: (f,l)\in Fr(E)\times L(F) \mapsto f \circ l \circ f^{-1}.$$
\end{Definition}
\begin{Proposition}
For the diffeology $\p_{L,Aut},$
	let $$D_L = \coprod_{x \in X} E_x \times End(E)_x$$ equipped with its subset diffeology in $E \times End(E).$ Then the evaluation map
	$$ ev : (x,e) \in D_L\mapsto ev_x(e) = e(x)$$
	is smooth.
\end{Proposition}
\begin{proof}
		In order to check the smoothness of the evaluation map, it is sufficient to show that $$(x,f,l) \in E \times Fr(E) \times L(F) \mapsto f \circ l \circ f^{-1} (x)$$ is smooth, which is true since composition is smooth.  
\end{proof}
Therefore, one has two possible generalizations of the notion of connection to a diffeological pseudobundle: 
\begin{itemize}
	\item either connection 1-forms associated to the $GL(F)-$principal bundle which are $GL(F)-$covariant 1-forms in $\Omega^1(Fr(E),\mathfrak{gl}(F)),$ and we note this set of connections by $\mathcal{C}_{\mathfrak{gl}}(E),$
	\item or connection 1-forms that lie in $\Omega^1(M,End_L(E)).$
\end{itemize}
The space $\mathcal{C}_{\mathfrak{gl}}(E)$ seems to be the most natural for a straightforward generalization of the theory of connections in a principal bundle, and we have to explain why we highlight also the second framework. For this, we have to propose a generalization of the notion of covariant derivative. But before that we need to find a way back to the initial diffeology $\p$ of $E,$ since all our constructions are performed till now in the weaker diffeology $\p_{Aut(E)}.$ Concerning this direction, we must say that our investigations have failed in the whole generality that we expected. We present here the sole way that we found convincing.
\subsection{Covariant deriivatives in the initial diffeology}
The point of the diffeology $\p_{Aut(E)}$  consists in making different connected components the domains where the fibers of $E$ are not isomorphic, while the initial diffeology $\p,$ which is stronger than $\p_{Aut(E)},$ may have non-isomorphic fibers in the same connected component. For this, se need to define first a diffeology $\p_D$ on $D_E$ which completes already defined. We require for this diffeology that 
\begin{itemize}
	\item for the canonical projection $$\pi_1: D_E= \coprod_{x \in X} E_x \times End(E)_x \rightarrow E$$ satisfies
	$(\pi_1)_* (\p_D) = \p$ 
	\item the evaluation map $ev:(D_E, \p_D) \rightarrow (E,\p)$ is smooth.
\end{itemize}
Therefore, a natural diffeology can be defined.
\begin{Definition}
We define $\p_D$ the diffeology on $D_E$ such that 
$$\p_D = ev^*(\p) \cap \pi_1^*(\p).$$
\end{Definition}
Then, a natural diffeology arises also for $End(E):$
\begin{Definition}
	Let $\p_L$ be the diffeology on $End(E)$ defined by $$\p_L=(\pi_2)_*(\p_D),$$ where $\pi_2$ is the second coordinate projection $$\pi_2: D_E=\coprod_{x \in X} E_x \times End(E)_x \rightarrow End(E).$$
\end{Definition}
Obviously, $End(E)$ is a vector pseudo-dundle which fibers are all diffeological algebras, but we can state a little more:
\begin{Proposition}
	Addition and multiplication are smooth maps from $End(E)^{(2)}$ to $End(E).$
\end{Proposition}
\begin{proof}
	Follows from the smoothness of addition in $E.$
\end{proof}
\begin{Definition} \label{CD}
	A covariant derivative $\nabla$ on $E$ is a collection of covariant derivatives $\left(p^*\nabla\right)_{p \in \p}$ such that 
	\begin{itemize}
		\item 	$\forall p \in \p$ with domain $O_p,$ $p^*\nabla$ is a covariant derivative on $p^*E,$ that is, $$p^*\nabla: TO_p \times \Gamma(O_p,p^*E) \rightarrow p^*E$$  is a smooth fiberwise bilinear map.
		\item $\forall p \in \p$ with domain $O_p,$ $\forall f \in C^\infty_c(O_p,\R), \forall X,Y \in TO_p \times \Gamma(O_p,p^*E), $
		$$ p^*\nabla_{fX}Y = f p^*\nabla_X Y $$
		and 
		$$ p^*\nabla_{X}(fY) = D_xf Y + f p^*\nabla_X Y.$$
		\item if $(p,p')\in \p^2,$ with domains $O_p$ and $O_{p'}$ such that there is a smooth map $g: O_{p'} \rightarrow O_p$ that satisfies $p' = p \circ g$ then $$(p')^*\nabla = g^*(p^* \nabla).$$  
	\end{itemize}
	We note by $\mathcal{CD}(E)$ the space of covariant derivatives on $E.$
\end{Definition}
\begin{Theorem}
	If $\mathcal{CD}(E)$ is non empty, then $\mathcal{CD}(E)$ is a diffeological affine space modelled on $\Omega^1(M,E).$
\end{Theorem}
\begin{proof}
	The proof follows the standard proof for covariant derivatives in the category of (finite dimensional) vector bundles. We evaluate the conariant derivatives on each plot $p \in \p.$ 
	Let $\nabla^{1}$ and $\nabla^{(2)}$ be two covariant derivatives. Then,  $\forall p \in \p$ with domain $O_p,$ $\forall f \in C^\infty_c(O_p,\R), \forall X,Y \in TO_p \times \Gamma(O_p,p^*E), $ $$p^*\nabla^{(1)}_X (f Y) - p^*\nabla^{(2)}_X (f Y) = f p^*\nabla_X Y,$$
	and since $O_p$ is an open subset of an Euclidian space, it is sufficient to show that $p^*\nabla^{(1)}_X  - p^*\nabla^{(2)}_X $  is a fiberwise multiplication operator. Therefore, $p^*\nabla^{(1)}  - p^*\nabla^{(2)} \in \Omega^1(X,E).$
	
	Conversely, let $\alpha \in \Omega^1(X,E),$ let $\nabla \in \mathcal{CD}(E),$ then we have analyze $\nabla' = \nabla + \alpha.$ A straightforward computation on each $p \in \p$ shows that the desired relations are fulfilled.
\end{proof}
 
 \section{Two examples}
 \subsection{On infinite dimensional vector bundles}
 It is quite straightforward that Definition \ref{CD} coincides with the classical definition of a covariant derivative on a finite dimensional vector bundle. Let us now analyze how it fits or differs with the existing definitions of covariant derivatives for infinite dimensional settings. Before that, we have to highlight the problem of \emph{existence}, not only existence of at least one covariant derivative on an infinite dimensional vector bundle $E,$ but also existence of at least one non-identically null global section $X:M \rightarrow E,$ can be proven in the actual state of knowledge through the existence of a smooth partition of the unit subordinate to a system of local trivializations of a vector bundle. To our best knowledge, these conditions can be gathered in the ILH setting \cite{Om} along the lines of \cite{Br}. This aspect will not be discussed here. 
 
 A second problem that we have to highlight is the problem of smoothness of the evaluation of linear maps. Following the review \cite{Gl2022}, when $F$ is a locally convex vector space, and unless $F$ is a Banach space, 
 the evaluation map $$L(F) \times F \rightarrow F$$ is merely hypocontinuous. This leads to the definition of infinite dimensional vector bundle \emph{with structure group} where a prescribed Lie group $G$ with Lie algebra $\mathfrak{g}$ is such that $G$ and $\mathfrak{g}$ are acting smoothly on $F.$ With this assumption, the construction of a principal bundle of $G-$frames, a bundle of $G-$endomorphisms, a space of $\mathfrak{g}-$valued connection 1-forms and their $G-$covariant derivatives are safely defined, see e.g. \cite{Ma2006}. Let us now develop a remarrk from \cite{GMW2023}.
 For this, we assume that $F$ is a diffeological vector space with field of scalars a diffeological field $\mathbb{K}.$ 
 \begin{Proposition}
 	\begin{enumerate}
 		\item The space $\mathcal{L}(F)$ of smooth linear maps from $F$ to $F$ is an algebra {with smooth addition, smooth multiplication and smooth scalar multiplication.}
 		\item The group of invertible elements of $\mathcal{L}(F),$ that we note by $GL(F),$ is a diffeological subgroup of $\operatorname{Diff}(F).$ 
 	\end{enumerate}
 \end{Proposition}
 \begin{rem}
 	The space $L(F)$ is called a diffeological algebra with group of the units $GL(F).$ In this setting, $GL(F)$ is \textbf{not} equipped with the subset diffeology in $L(F).$ Its diffeology needs to make the inverse map smooth, which requires to consider a restricted diffeology that makes smooth the mapping $$ (.)^{-1} :  f \in GL(F) \mapsto f^{-1}\in GL(F).$$
 	This feature is well-known in the construction of the diffeology of the group of diffeomorphisms of a diffeological space.
 \end{rem}
 When $E$ and $F$ are locally convex complete topological vector spaces, the nebulae diffeologies $\p_\infty(E)$ and $\p_\infty(F)$ are well defined and the space of smooth linear maps $L(E,F)$ (in the classical sense, for G\^ateaux derivation) is alsoa  space of smooth maps in the diffeological sense. Therefore, we get (diffeological) smoothness for the evaluation map $$\mathcal{L}(E,F) \times E \rightarrow F$$
and also the smoothness of left and right compositions $$(g,l,g') \in GL(F) \times L(E,F) \times GL(E) \mapsto g \circ l \circ g' \in L(E,F).$$ 
In order to be able to describe the full picture of connections and covariant derivatives, we need to complete trivial statements by the following result, which is not optimal in the sense of generality but at least optimal in its simplicity of use in a wide class of examples:
\begin{Theorem}
	If $L(F)$ is separated by one-forms on $F,$ then $GL(F)$ is a diffeological Lie group.
\end{Theorem}
 \begin{proof}
 	Since $F$ is complete, we have that $T_{Id}GL(F) \subset L(F).$ Following \cite[Theorem 1.14]{Les}, considering the one forms $\alpha$ on $T_{Id}GL(F)$ defined by $\alpha = \beta \circ (.)$ where $\beta$ is a 1-form on $F,$ we get a family which separates $L(F)$ and hence $T_{Id}GL(F).$ Therefore, $GL(F)$ is a diffeological Lie group.
 \end{proof}

Till the end of this section, when we write $\mathfrak{gl}(F)$ instead of $T_{Id}GL(F)$  we assume that $G$ is a diffeological Lie group, and therefore that $\mathfrak{gl}(F)$ is a Lie algebra which is a diffeological vector subspace of $L(F),$ not necessarily with its subset diffeology, but with its diffeology of internal tangent space.

\begin{Theorem}
	If $E$ is a fiber bundle with typical fiber $F,$ a complete locally convex vector space separated by a family of 1-forms on $F,$ then \begin{itemize}
		\item the group $GL(F)$ is a structure group for $E.$ Therefore, there exists a frame bundle $Fr_F(E)$ which is a principal $GL(F)-$bundle. The connection 1-forms on $Fr_F(E)$ define a class $\mathcal{CD}_{GL}$ of covariant derivatives on $E.$
		\item  $\mathcal{CD}_{GL} \subset \mathcal{CD}$ and, if $\mathfrak{gl}(F) = L(F)$ as diffeological vector spaces, then $\mathcal{CD}_{GL} = \mathcal{CD}.$
	\end{itemize}
\end{Theorem} 
\begin{proof}
	Given a system of local trivialization $(\phi_i,\psi_i)_{i \in I}$ where $\{\psi_i\}_{i \in I}$ is an atlas on $M$ and  $$\left\{\phi_i:\psi_i^{-1}U_i \times F \rightarrow E \, | \, i \in I\right\}$$ produces the local trivializations.
	The local changes of trivializations $$\phi_i^{-1} \circ \phi_j: \psi_j^{-1}(U_i \cap U_j)\times F \rightarrow \psi_i^{-1}(U_i \cap U_j)\times F $$ define transtition maps with values in $GL(F)$ and therefore $E$ is a $GL(F)-$bundle. Therefore, following \cite{KM,Ma2006}, one produces a frame bundle, its connection $1-$forms, and then its covariant derivatives which is an affine space modelled on 1-forms with values in $\mathfrak{gl}(F)-$endomorphisms, which ends the first point.  
	The second point remains on the fact that $L(F) \subset \mathfrak{gl}(F).$
\end{proof}
We finish this example by the following remark: the frame bundle that we have constructed does relie on the group of automorphisms of $E,$ but not on the diffeology $\p_{Aut(E)}.$
Indeed, in this construction, the principal bundle $Fr_F(E)$ fully generates 
${GL}(E) $ by the surjective maps
$$ \forall x \in X, (f_1,f_2) \in Fr_F(E) \mapsto f_1 \circ f_2^{-1},$$
while the (full) diffeological  groups $Aut(E)$ and $Diff(X)$ remain poorly understood to our knowledge.

 Therefore the principal bundle $Fr_F(E)$ does not coincide with the principal bundle $Fr(E)$ defined in Definition \ref{frame-aut} by its diffeology $\p_{Aut(E)}$ which is actually unknown. However, this definition remains valid on infinite dimensional vector bundles, which group of automorphisms needs to be studied clearly.
 \subsection{Application to algebraic varieties}
 	Let $P \in \R[X_1,X_2,...X_n]$ and let us consider the algebraic variety $M = Ker P \subset \R^n$ equipped with its subset diffeology. This diffeology is a diffeology of Fr\"olicher space, but one can notice that this Fr\"olicehr structure does not have all its smooth functions that are restrictions of smooth functions on $\R^2,$ even if $M$ is closed. The algebraic variety $M$ is a diffeological space and its various tangent spaces, as well as its cotangent space and its space of differential forms, can be vector pseudo-bundles.
We note by $\mathcal{AV}(n)$ the set of algebraic varieties on $\R^n.$  We compare the various tangent spaces, the corresponding frame bundles and the covariant derivatives around two examples.

\begin{example}
	The example of $P(X,Y) = XY$, and $M = P^{-1}(0) \subset \R^2$ is discussed extensively in e.g. \cite{Ma2020-3} in terms of tangent spaces. The Diff-diffeology separates $M$ into five connected components, one isomorphic to one point and for isomorphic to $\R.$ Therefore, the diff-tangent space, as well as the internal tangent space and the external tangent space, noted independently by $E$ in the rest of the example, have the same frame bundle, and connection 1-forms are null at $(0,0),$ and may not connect even topologically (for the initial topology) with 1-forms on the other connected components of the Diff-diffeology. Therefore, there is and identification of $End_L(E)$ with $\mathcal{C}_{\mathfrak{gl}}(E)$ (remember that, from Definition \ref{EndLE}, the diffeology of $End_L(E)$ is derived from $\p_{Aut(E)}$). We can prove that the diffeology of $End(E)$ differ from their diffeology, and hence $\mathcal{CD}_{GL}$ and $\mathcal{CD}$ do not have the same model vector space. A direct computation shows that $\mathcal{CD}$is isomorphic (as a diffeological space) to $(\Omega^1(\R,\R))^2$ while $\mathcal{CD}_{GL}$ is diffeologically isomorphic with $(\Omega^1(\R,\R))^4$ with no geometrically meaningful canonical map between the two spaces.    
\end{example}  
\begin{example}
	Let $P(X,Y,Z)= X^2 + Y^2 - Z^2$ and let $M = P^{-1}(0) \subset \R^3$ be the standard cone. The same kind of considerations as in the previous example is valid: the Diff-diffeology separtes $M$ into three manifods, two of them are $2-$dimensional and isomorphic to $R\times S^1,$ while the last connected component is one point. With the same arguments, $\mathcal{CD}_{GL}$ is diffeologically isomorphic to $(\Omega^1(S^1 \times \R, \mathfrak{gl}(\R^2)))^2$ while $\mathcal{CD}$ is isomorphic to the set of $\mathfrak{gl}(\R^2)-$valued 1-forms $\theta$ such that, at $(x,y) \in \R \times S^1$ such that $y = 0,$  \begin{itemize}
		\item $\alpha_{(x,y)}(\partial_y)=0$
		\item $\alpha_{(x,y)}(\partial_x) $ is a projection is a map with values in $span(\partial_x).$
	\end{itemize} 
\end{example}
	\section{Outlook}
	The structures that we highlight here seem to be present but forgotten in more standard settings. Indeed, 
	\begin{itemize}
		\item the existence of the group $G_E(X),$ as a subgroup of $Diff(X),$ is already known for classical finite dimensional vector bundles see e.g. \cite{Neeb2008}, but it seems to be defined and not described in the common litterature. In the more general setting of vector pseudo-bundles, we can show here, through Lemma  \label{fiber}, that $G_E(X) \neq Diff(X).$ Moreover, through the existence of the diffeology $\p_{Aut(E)},$ one can guess that the following open question needs an answer: 
		
		\textbf{Open question:} in which case and for which diffeologies do we have $T_{Id_M}G_E(X)= T_{Id_M} Diff(X)$ and $T*_{Id_M}G_E(X)= T^*_{Id_M} Diff(X)?$
		\item the classical frame bundle $Fr(E),$ that we have constructed along the lines of the theory of principal bundles associated to finite dimensional fiber or vector bundles can be extended in the diffeological vector pseudo-bundles setting, seemingly, only with respect to the ``restricted'' diffeology $\p_{Aut(E)}$ and not with respect to the diffeology $\p$ of $E.$
		\item the notion of covariant derivative then seems somewhat disconnected to the notion of connection 1-form for diffeological vector pseudo-bundles because we showed how one can naturally extend the classical notion of a covariant derivative taking into account the diffeology $\p$ of $E.$  
	\end{itemize}
Therefore, the extension of the classical geometry of covariant derivatives along the diffeological framework here described seems to be possible. Future works will help us to decide how far. 
	
\end{document}